\documentclass[11pt]{amsart}

\usepackage{epigamath}

\usepackage[english]{babel}

\numberwithin{equation}{section}

\usepackage{amsthm, amsfonts,amsmath,amscd,amssymb}
\usepackage{mathrsfs}
\usepackage{mathtools}
\usepackage{tikz-cd}
\usepackage{setspace}

\newtheorem{thm}{Theorem}[section]
\newtheorem*{thm*}{Theorem}
\newtheorem{lemma}[thm]{Lemma} 
\newtheorem{prop}[thm]{Proposition} 
\newtheorem{cor}[thm]{Corollary}

\theoremstyle{definition}
\newtheorem{defn}[thm]{Definition}

\theoremstyle{remark}

\newtheorem{remark}[thm]{Remark}

\newcommand\HH{\mathrm{H}}

\newcommand{\OO}{\mathcal{O}}

\DeclareMathOperator{\et}{\acute{e}t}

\DeclareMathOperator{\Gal}{Gal}

\DeclareMathOperator{\GL}{GL}

\DeclareMathOperator{\SL}{SL}
\DeclareMathOperator{\Aut}{Aut}
\DeclareMathOperator{\Stab}{Stab}
\DeclareMathOperator{\Nrd}{Nrd}
\DeclareMathOperator{\Sp}{Sp}

\DeclareMathOperator{\rank}{rank}

\DeclareMathOperator{\Pic}{Pic}

\newcommand{\bC}{{\mathbb C}}

\newcommand{\bF}{{\mathbb F}}

\newcommand{\bG}{{\mathbb G}}

\newcommand{\bP}{{\mathbb P}}
\newcommand{\bQ}{{\mathbb Q}}

\newcommand{\sA}{{\mathcal A}}

\newcommand{\sD}{{\mathcal D}}

\newcommand{\sH}{{\mathcal H}}

\newcommand{\sL}{{\mathcal L}}
\newcommand{\sM}{{\mathcal M}}

\newcommand{\fX}{{\mathfrak X}}

\newcommand{\bZ}{\ensuremath{\mathbb{Z}}}

\def\lowsim{\vbox to 0pt{\vss\hbox{$\scriptstyle\sim$}\vskip-1.5pt}}

\newcommand{\longhookrightarrow}{\xhookrightarrow{\hphantom{aaa}}}
\DeclareRobustCommand\longtwoheadrightarrow
    {\relbar\joinrel\twoheadrightarrow}

\newcommand{\supth}[1]{\ensuremath{#1^{\mathrm{th}}}}
\newcommand{\supst}[1]{\ensuremath{#1^{\mathrm{st}}}}

\EpigaVolumeYear{8}{2024} \EpigaArticleNr{16} \ReceivedOn{August 8, 2023}
\InFinalFormOn{March 31, 2024}
\AcceptedOn{May 6, 2024}

\title{Line bundles on the first Drinfeld covering}
\titlemark{Line bundles on the first Drinfeld covering}

\author{James Taylor}
\address{Centre for Mathematical Sciences, University of Cambridge, Cambridge, CB3 0WA, UK}
\email{jt950@cam.ac.uk}

\authormark{J.~Taylor}

\AbstractInEnglish{Let $\Omega^d$ be the \mbox{$d$-dimensional} Drinfeld symmetric space for a finite extension $F$ of~$\bQ_p$. Let $\Sigma^1$ be a geometrically connected component of the first Drinfeld covering of $\Omega^d$, and let $\bF$ be the residue field of the unique degree $d+1$ unramified extension of $F$. We show that the natural homomorphism
	\[
	\widehat{(\bF, +)}\longrightarrow \Pic(\Sigma^1)[p]
	\]
	determined by the second Drinfeld covering is injective.
	Here $\widehat{(\bF, +)}$ is the group of characters of $(\bF, +)$.
	In particular, $\Pic(\Sigma^1)[p] \neq 0$. We also show that all vector bundles on $\Omega^1$ are trivial, which extends the classical result that $\Pic(\Omega^1) = 0$.}


\acknowledgement{This research was financially supported by the EPSRC.}

\begin{document}


\maketitle

\begin{prelims}

\DisplayAbstractInEnglish

\bigskip

\DisplayKeyWords

\medskip

\DisplayMSCclass

\end{prelims}


\newpage

\setcounter{tocdepth}{1}

\tableofcontents


\section{Introduction}

Let $p$ be a prime, $F$ a finite extension of $\bQ_p$, and $L$ the completion of the maximal unramified extension of $F$. The \emph{Drinfeld tower} is a system of $d$-dimensional rigid analytic spaces over $L$,
\[
	\sM_0 \longleftarrow \sM_1 \longleftarrow \sM_2 \longleftarrow \cdots,
\]
for which the spaces $\sM_n$ are equipped with compatible actions of $D^\times \times \GL_{d+1}(F)$, where $D$ is the division algebra with invariant $1/(d+1)$ over $F$; see \cite{DRI, BC, RZ}. The space $\sM_0$ is a fundamental example of a Rapoport--Zink space, and as such is defined as the generic fibre of a $p$-adic formal scheme $\widehat{\sM_0}$ which parametrises special formal $\OO_D$-modules. The spaces $\sM_n$ over $\sM_0$ are obtained by considering level structure by the compact open subgroups $1 + \Pi^n \OO_D$ of $D^\times$, and each $\sM_n \rightarrow \sM_0$ is a finite \'{e}tale Galois covering with Galois group $\OO_D^\times / (1 + \Pi^n \OO_D)$. These spaces play an important role in the representation theory of both $\GL_{d+1}(F)$ and $D^\times$: this tower has been shown to realise both the local Langlands and the Jacquet--Langlands correspondences for $\GL_{d+1}(F)$ in its \'{e}tale cohomology, see~\cite{CAR, HARR, BOY, HARTAY}, and when $F = \bQ_p$ and $d = 1$, encode part of the $p$-adic local Langlands correspondence for $\GL_2(\bQ_p)$; see~\cite[Theorem~0.2]{CDN1}.

The base space $\sM_0$ is non-canonically a disjoint union over $\bZ$ of copies of $\Omega^d$, the $d$-dimensional Drinfeld symmetric space. This is the admissible open subset of $\bP^{d, \text{an}}$ defined by removing all $F$-rational hyperplanes. The space $\sM_1$ has been studied by many authors, see \cite{TEIT, WANG, LP, JUNEQN}, and recently Junger has shown that, after extending from $L$ to $L(\varpi)$, the preimage of each copy of $\Omega^d$ in $\sM_1$ is a disjoint union of $q-1$ copies of $\Sigma^1$, a particular geometrically connected Kummer-type cyclic Galois covering of $\Omega^d$; see \cite[Theorem~4.9]{JUNEQN}. Here $\varpi$ is a $\supst{(q-1)}$ root of $-\pi$, for $\pi$ a uniformiser of $F$, and the extension $L(\varpi) / L$ is equal to the first Lubin--Tate extension of $L$. Little is known about the geometry of the higher covering spaces $(\sM_n)_{n \geq 2}$.

It is a classical result that $\Pic(\Omega^1) = 0$, and recently this has been generalised to higher dimensions and more general hyperplane arrangements by Junger; see \cite[Theorem~A]{JUNCOH}. Understanding the Picard groups of the covering spaces $(\sM_n)_{n \geq 1}$ is much more difficult, and almost nothing is known in this context. Previously, we showed that for $d = 1$, there is no $p$-torsion in the Picard group of any open subset of $\Sigma^1$ which is the preimage of a vertex of the Bruhat--Tits tree; see \cite[Theorem~3.2]{JT}.

In this paper we show that in any dimension $\Pic(\Sigma^1)[p] \neq 0$. More precisely, setting $G \coloneqq \SL_{d+1}(\OO_F)$ and writing $\bF$ for the residue field of the unique degree $d+1$ unramified extension of $F$, we prove the following.

\begin{thm*}[Theorem~\ref{mainthm}]
Suppose that $K$ contains $L(\varpi)$ and a primitive $\supth{p}$ root of\, $1$. Then the natural homomorphism
\[
	\widehat{(\bF, +)} \rightarrow \Pic(\Sigma^1)[p]^G
\]
determined by the second Drinfeld covering is injective.
\end{thm*}

Here $\widehat{(\bF, +)}$ is the group of characters of $(\bF, +)$. The assumption that $K$ contains $L(\varpi)$ is simply to ensure that the space $\Sigma^1$ is defined, and the assumption that $K$ contains a primitive $\supth{p}$ root of $1$ is to ensure that the homomorphism, which we now describe, is defined. For any Galois covering $f \colon X \rightarrow Y$ with abelian Galois group $H$ of exponent $e$ over a field which contains a primitive $\supth{e}$ root of $1$, there is a decomposition
\[
	f_* \OO_X = \bigoplus_{\chi \in \widehat{H}} \sL_{\chi}, \quad \sL_{\chi} \coloneqq e_{\chi} \cdot f_* \OO_X,
\]
where $e_{\chi}$ is the central primitive idempotent corresponding to the character $\chi$. Furthermore, $\sL_{\chi} \in \Pic(Y)[e]$ for any $\chi \in \widehat{H}$, and the association
\[
	\widehat{H} \longrightarrow \Pic(Y)[e], \quad \chi \longmapsto \sL_{\chi}
\]
is a group homomorphism (\textit{cf.} Proposition~\ref{abeliangaloisprop}). For any $n \geq 1$, $\sigma_n \colon \sM_{n+1} \rightarrow \sM_{n}$ is a Galois covering with Galois group $(1 + \Pi^n \OO_D) / (1 + \Pi^{n+1}\OO_D)$, which is canonically identified with $(\bF, +)$. The homomorphism of Theorem~\ref{mainthm} is then the homomorphism associated to the abelian Galois covering $\sigma_1 \colon \Sigma^2 \rightarrow \Sigma^1$, where $\Sigma^2$ is the preimage of $\Sigma^1$ in $\sM_2$.

Now for the remainder of the introduction we focus on the case where $d = 1$ and set $\Omega \coloneqq \Omega^1$. Our main interest in Theorem~\ref{mainthm} is in the following. The work~\cite{DLB} of Dospinescu and Le Bras describes how certain $D^\times$-isotypical parts of the locally analytic representations $\OO(\sM_n)'$ of $\GL_2(\bQ_p)$ are related to the $p$-adic Langlands correspondence and to the Jacquet--Langlands correspondence. We would like to better understand the representations $\OO(\sM_n)'$ more generally for any $F$ and $d \geq 1$. For the representation $\OO(\sM_1)'$ and questions regarding admissibility and topological irreducibility, because $\sM_1$ is a disjoint union of copies of $\Sigma^1$, it is sufficient to understand the $D(G)$-module $\OO(\Sigma^1)$, where $D(G)$ is the distribution algebra of $G$.

The extension $\sigma_0 \colon \Sigma^1 \rightarrow \Omega$ is an abelian Galois covering with Galois group $\Gamma \coloneqq (\bF_{q^2}^\times)^{q-1} \subset \bF_{q^2}^\times$. Applying the above formalism to this extension and then taking the global sections, we can decompose the $D(G)$-module $\OO(\Sigma^1)$  as 
\[
	\OO(\Sigma^1) = (\sigma_{0, *} \OO_{\Sigma^1})(\Omega)= \bigoplus_{\psi \in \widehat{\Gamma}} \sL_{\psi}(\Omega).
\]
The monumental recent work of Ardakov and Wadsley \cite{AW} shows, without any restriction on $F$, that if $\psi \neq 1$, then $\sL_{\psi}(\Omega)$ is a topologically irreducible coadmissible $D(G)$-module. This follows from~\cite[Theorem~A]{AW} together with the proof of~\cite[Corollary~B]{AW} and the fact that $\Sigma^1$ is geometrically connected. The case where $\psi = 1$ is simply $\sL_{1}(\Omega) = \OO(\Omega)$, and the authors additionally show that this has length~$2$ as a $D(G)$-module; see \cite[Corollary~7.5.9]{AW}. We note that unlike $\sL_{\psi}$ for $\psi \neq 1$, the line bundle $\sL_1 = \OO_{\Omega}$ is the restriction of some $\GL_2(F)$-equivariant vector bundle on $\bP^{1, \text{an}}$ to $\Omega$, and as such was already well understood as a $D(\GL_2(F))$-module; see \cite{ORL}. The main idea of \cite{AW} is to understand the $G$-equivariant line bundle with connection $\sL_{\psi}$ on $\Omega$ and then push this to $\bP^{1,\text{an}}$ and use the $G$-equivariant Beilinson--Bernstein correspondence, see \cite[Theorem~C]{ARD}, to deduce properties of the corresponding $D(G)$-module $\sL_{\psi}(\Omega)$. 

We would like to use similar techniques to understand the global sections of the higher Drinfeld coverings as $D(G)$-modules. For the second covering, we consider the extension $\sigma_1 \colon \Sigma^2 \rightarrow \Sigma^1$ and want to understand the $G$-equivariant line bundles with connection $\sL_{\chi} = e_{\chi} \cdot \sigma_{1,*}\OO_{\Sigma^2}$ for any $\chi \neq 1$. In this context Theorem~\ref{mainthm} says that the underlying line bundle of each $\sL_{\chi}$ is non-trivial. This is in contrast to what happens for $\sigma_0 \colon \Sigma^1 \rightarrow \Omega$, because $\Pic(\Omega) = 0$. The analysis used in~\cite{AW} of $G$-equivariant line bundles with connection works under the assumption that the underlying line bundle is trivial, and Theorem~\ref{mainthm} shows that none of the line bundles $\sL_{\chi}$ fit into this formalism.

In this paper we also show that all vector bundles on $\Omega$ are trivial (see Corollary~\ref{cor2}), which extends and uses the classical result that $\Pic(\Omega) = 0$. The key property we use is that $\OO(\Omega)$ is a Pr\"{u}fer domain, which is unknown to hold in higher dimensions. In the context of the above discussion, this suggests that rather than first understanding $\sL_{\chi} = e_{\chi} \cdot \sigma_{1,*}\OO_{\Sigma^2}$ as a $G$-equivariant line bundle with connection on $\Sigma^1$ and then pushing to $\Omega$, a potentially more feasible approach is to instead analyse $\sigma_{0,*} \sL_{\chi}$ directly as a $G$-equivariant vector bundle with connection on $\Omega$.

\subsection*{Notation} 

Throughout we shall use the following notation:  $F$ is a finite extension of $\bQ_p$, with ring of integers $\OO_F$, uniformiser $\pi$, and residue field $\bF_q$; $K$ is a complete field extension of $F$, $L$ is the completion of the maximal unramified extension of $F$, and $\bC_p$ is the completion of $\overline{F}$. The integer $d \geq 1$ will denote the dimension of the spaces we consider. We set $G = \SL_{d+1}(\OO_F)$ and write $D$ for the division algebra over $F$ of invariant $1/(d+1)$.

\subsection*{Acknowledgments} The author would like to thank Konstantin Ardakov for many useful discussions and both Konstantin Ardakov and Simon Wadsley for making available an early version of their preprint~\cite{AW}. The author would also like to thank the referee for their time and comments, all of which improved the paper.

\section{Abelian Galois coverings}\label{section1}

In this section we describe how the pushforward of the structure sheaf of an abelian Galois covering decomposes into line bundles. The approach we take here is influenced by the work  of Borevi\v{c} for Kummer extensions of rings; see \cite{BOR}.

Let $\Gamma$ be an abstract group. We write $\underline{\Gamma}$ for the corresponding constant rigid analytic group over $K$. Recall that  a (right) action of $\underline{\Gamma}$ on a rigid space $X$ over $K$ is a morphism $a \colon X \times \underline{\Gamma} \rightarrow X$ of rigid spaces over $K$ such that the diagrams
\[\begin{tikzcd}
	{X \times \underline{\Gamma} \times \underline{\Gamma}} & {X \times \underline{\Gamma}} && {X \times \underline{1}} && {X \times \underline{\Gamma}} \\
	{X \times \underline{\Gamma}} & X\rlap{,} &&& X
	\arrow["{a \times p_2}", from=1-1, to=1-2]
	\arrow["{p_X \times m}"', from=1-1, to=2-1]
	\arrow["a"', from=2-1, to=2-2]
	\arrow["a", from=1-2, to=2-2]
	\arrow["{p_X}"', from=1-4, to=2-5]
	\arrow["a", from=1-6, to=2-5]
	\arrow[hook, from=1-4, to=1-6]
\end{tikzcd}\]
commute. If in addition $f \colon X \rightarrow Y$ is a morphism of rigid spaces over $K$, then the action is called equivariant with respect to the trivial action of $\underline{\Gamma}$ on $Y$ if
\[\begin{tikzcd}
	{X \times \underline{\Gamma}} & X \\
	X & Y
	\arrow["{p_X}"', from=1-1, to=2-1]
	\arrow["a", from=1-1, to=1-2]
	\arrow["f"', from=2-1, to=2-2]
	\arrow["f", from=1-2, to=2-2]
\end{tikzcd}\]
commutes.
As for schemes, an action of $\underline{\Gamma}$ on $X$ which is equivariant with respect to the trivial action of $\underline{\Gamma}$ on $Y$ is equivalent to the data of a group homomorphism $\rho \colon \Gamma^{\text{op}} \rightarrow \Aut_Y(X)$, where $\Aut_Y(X)$ is the group of automorphisms of $X$ which respect the morphism $f \colon X \rightarrow Y$.
In this situation the sheaf of $\OO_Y$-modules $f_* \OO_X$ has a (left) action of $\Gamma$,
\[
	\Gamma \longrightarrow \Aut_k\left(f_*\OO_X\right), \quad g \longmapsto \left(\rho(g)^\sharp_{f^{-1}(U)} \colon f_*\OO_X(U) \longrightarrow f_*\OO_X(U)\right)_{U \subset Y},
\]
which is well defined as $\rho(g)^\sharp \colon \OO_X \rightarrow \rho(g)_* \OO_X$ and
\[
\rho(g)^{-1}(f^{-1}(U)) = f^{-1}(U)
\]
for any open $U \subset Y$. Therefore, we can consider the sheaf of $\OO_Y$-modules $(f_* \OO_X)^{\Gamma}$ defined by
\[
	(f_* \OO_X)^{\Gamma}(U) = \OO_X(f^{-1}(U))^{\Gamma}
\]
for any admissible open subset $U$ of $Y$, which is a sheaf because $(-)^{\Gamma}$ preserves products and equalisers. 

\begin{defn}
Suppose that $\Gamma$ is a finite group, $a \colon X \times \underline{\Gamma} \rightarrow X$ is an action of $\underline{\Gamma}$ on $X$, and $f \colon X \rightarrow Y$ is a finite \'{e}tale morphism of rigid spaces over $K$ which is equivariant with respect to the trivial action of $\underline{\Gamma}$ on $Y$. Then $f \colon X \rightarrow Y$ is a \emph{Galois covering with Galois group $\Gamma$} if the natural map $\OO_Y \rightarrow (f_* \OO_X)^{\Gamma}$ is an isomorphism of $\OO_Y$-modules and
\[
	p_X \times a \colon X \times \underline{\Gamma} \longrightarrow X \times_Y X
\]
is an isomorphism of rigid spaces over $K$.
\end{defn}

For the remainder of this section, we assume that $\Gamma$ is a finite abelian group, $f \colon X \rightarrow Y$ is a Galois covering with Galois group $\Gamma$, and $K$ contains a primitive $\supth{e(\Gamma)}$ root of $1$, where $e(\Gamma)$ is the exponent of $\Gamma$. For each $\chi \in \widehat{\Gamma}$, we write $e_{\chi}$ for the corresponding central primitive idempotent
\[
	e_{\chi} = \frac{1}{|\Gamma|} \sum_{\gamma \in \Gamma} \chi(\gamma^{-1}) \gamma \in K[\Gamma].
\]

\begin{defn}
For any $\chi \in \widehat{\Gamma}$, we define the $\OO_Y$-module
\[
	\sL_{\chi} \coloneqq e_{\chi} \cdot f_* \OO_X.
\]
\end{defn}

\begin{prop}\label{abeliangaloisprop}
There is a direct sum decomposition of\, $\OO_Y$-modules
\[
	f_* \OO_X = \bigoplus_{\chi \in \widehat{\Gamma}} \sL_{\chi},
\]
and multiplication in $f_* \OO_X$ induces an isomorphism
\[
	\sL_{\chi} \otimes_{\OO_Y} \sL_{\psi} \stackrel{\lowsim}{\longrightarrow} \sL_{\chi \psi}.
\]
In particular, each $\sL_{\chi}$ is an $e(\Gamma)$-torsion invertible $\OO_Y$-module, and the association
\[
	\widehat{\Gamma} \longrightarrow \Pic(Y)[e(\Gamma)], \quad \chi \longmapsto \sL_{\chi}
\]
is a group homomorphism.
\end{prop}

\begin{proof}
The sheaf $f_*\OO_X$ is an $\OO_Y[\Gamma]$-module, and the direct sum decomposition of $f_*\OO_X$ follows from the fact that the $e_{\chi}$ are central orthogonal idempotents which sum to $1$. Now suppose that $U$ is an affinoid open subset of $Y$, and let
\[
V \coloneqq f^{-1}(U) = U \times_Y X \longhookrightarrow X.
\]
Then $\underline{\Gamma}$ acts on $V$, $f \colon V \rightarrow U$ is equivariant with respect to the trivial action of $\underline{\Gamma}$ on $U$, and we have a commutative diagram of isomorphisms
\[\begin{tikzcd}
	{U \times_Y (X \times \underline{\Gamma})} & {U \times_Y(X \times_Y X)} \\
	{V \times \underline{\Gamma}} & {(U \times_Y X) \times_U ( U \times_Y X)\rlap{.}}
	\arrow[from=1-2, to=2-2]
	\arrow[from=2-1, to=2-2]
	\arrow[from=1-1, to=2-1]
	\arrow[from=1-1, to=1-2]
\end{tikzcd}\]
Write $A \coloneqq \OO(U)$ and $B \coloneqq \OO(V)$. Since $f \colon X \rightarrow Y$ is finite, $V$ is affinoid, and because $\OO_Y \rightarrow (f_* \OO_X)^\Gamma$ is an isomorphism, $A \rightarrow B$ is injective and has image $B^\Gamma$.
Furthermore, because $B$ is finitely generated over $A$, the natural inclusion
\[
	B \otimes_A B \longrightarrow B \widehat{\otimes}_A B
\]
is an isomorphism; see \cite[Proposition~3.7.3/6]{BGR}.
Therefore, the composition of this inclusion with the global sections of $p_X \times a$ induces an isomorphism
\[
	B \otimes_A B \stackrel{\lowsim}{\longrightarrow}  B \otimes_K \OO(\Gamma), \quad x \otimes y \longmapsto \sum_{\gamma \in \Gamma} x(\gamma \cdot y) \otimes \delta_{\gamma}.
\]
Now $B$ is a right $\OO(\Gamma)$-comodule algebra for the Hopf algebra $\OO(\Gamma)$ via
\[
	\rho \colon B \longrightarrow B \otimes_K \OO(\Gamma), \quad \rho \colon b \longmapsto \sum_{\gamma \in \Gamma} \gamma(b) \otimes \delta_\gamma,
\]
and the above isomorphism says exactly that $A \rightarrow B$ is an $\OO(\Gamma)$-Galois extension in the sense of \cite[Definition~8.1.1]{MONT}. Because $K$ contains a primitive $\supth{e(\Gamma)}$ root of $1$, the natural map $K[\widehat{\Gamma}] \rightarrow \OO(\Gamma)$ is an isomorphism of Hopf algebras over $K$. Therefore, using this identification, we can view $B$ as a $K[\widehat{\Gamma}]$-comodule algebra. We have that for $b \in B$,
\begin{equation}\label{comodulemap}
	\rho(b) = \sum_{\gamma \in \Gamma} \gamma(b) \otimes \delta_{\gamma} = \sum_{\chi \in \widehat{\Gamma}} b_{\chi} \otimes \chi 
\end{equation}
for some unique $b_{\chi} \in B$, and for each $\chi \in \widehat{\Gamma}$, we define
\[
B_{\chi} \coloneqq \{b_{\chi} \mid b \in B \}.
\]
Because $A \rightarrow B$ is $K[\widehat{\Gamma}]$-Galois, these $B_{\chi}$ make $B$ a strongly graded $\widehat{\Gamma}$-algebra by a result of Ulbrich, see \cite[Theorem~8.1.7]{MONT}, meaning that 
\[
	B = \bigoplus_{\chi \in \widehat{\Gamma}} B_{\chi} \quad \text{and} \quad B_{\chi} \cdot B_{\psi} = B_{\chi \psi} \ \text{ for all }  \chi, \psi \in \widehat{\Gamma}.
\]
In fact, $e_{\chi} \cdot B = B_{\chi}$. Indeed, by column orthogonality
\[
	\delta_{\gamma} = \frac{1}{|\Gamma|} \sum_{\chi \in \widehat{\Gamma}} \chi(\gamma^{-1}) \chi,
\]
and therefore substituting this into Equation (\ref{comodulemap}) and comparing the coefficient of $\chi$ shows that
\[
	e_{\chi} \cdot b = b_{\chi}.
\]
There is a natural surjective morphism of $A$-modules
\[
	m_{\chi, \psi} \colon B_{\chi} \otimes_A B_{\psi} \longrightarrow B_{\chi} \cdot B_{\psi} = B_{\chi \psi}.
\]
In order to show that this is injective, first note that each $B_{\chi}$ is direct summand of $B$ and thus is finitely generated projective as an $A$-module. From the decomposition of $B$ as the direct sum of the $B_{\chi}$,
\begin{equation}\label{rankequality}
\sum_{\chi \in \widehat{\Gamma}} \rank_A(B_{\chi}) = \rank_A(B) = |\Gamma|.
\end{equation}
On the other hand, we have a surjection
\[
B_{\chi} \otimes_A B_{\chi^{-1}} \longrightarrow B_{\chi} \cdot B_{\chi^{-1}} = B_1 = e_1 \cdot B = A,
\]
and therefore $\rank_A(B_{\chi}) \rank_A(B_{\chi^{-1}}) \geq 1$, hence $\rank_A(B_{\chi}) \geq 1$, and from Equation (\ref{rankequality}) above, $\rank_A(B_{\chi}) = 1$. As a consequence, $m_{\chi, \psi}$ is a surjective homomorphism between finitely generated rank $1$ $A$-modules, and as such it is injective.

Now, returning to the global situation, multiplication induces a morphism of $\OO_Y$-modules 
\begin{equation}\label{OYmodhom}
	\sL_{\chi} \otimes_{\OO_Y} \sL_{\psi} \longrightarrow f_*\OO_Y.
\end{equation}
Because $f_*\OO_Y$ is coherent, then locally over an affinoid open subset $U$ as above, this is identified with the morphism of sheaves associated under the associated sheaf construction to the $A$-module homomorphism
\[
	e_{\chi} \cdot B \otimes_{A} e_{\psi} \cdot B \longrightarrow B.
\]
We have shown above that this has image $e_{\chi \psi} \cdot B$, and therefore the morphism of $\OO_Y$-modules (\ref{OYmodhom}) above induces an isomorphism
\[
	\sL_{\chi} \otimes_{\OO_Y} \sL_{\psi} \stackrel{\lowsim}{\longrightarrow} \sL_{\chi \psi}. \qedhere
\]
\end{proof}

\begin{remark}
In fact one can show that if $Y$ is connected and $\Gamma_0$ is the stabiliser of any connected component $X_0$ of $X$, then $f \colon X_0 \rightarrow Y$ is a Galois extension with Galois group $\Gamma_0$, and the homomorphism
	\[
		\widehat{\Gamma} \longrightarrow \Pic(Y)[e(\Gamma)]
	\]
	factors as the composition
	\[
		\widehat{\Gamma} \longtwoheadrightarrow  \widehat{\Gamma_0} \longrightarrow \Pic(Y)[e(\Gamma_0)] \longhookrightarrow \Pic(Y)[e(\Gamma)].
	\]
\end{remark}

\section{Drinfeld symmetric spaces}

Let $F$ be a finite extension of $\bQ_p$, $L$ the completion of the maximal unramified extension of $F$, and $K$ a complete field extension of $F$. Set $G \coloneqq \SL_{d+1}(\OO_F)$, and let $D$ be the division algebra over $F$ of invariant $1/(d+1)$ with ring of integers $\OO_D$. We let $\Pi$ denote a uniformiser of $\OO_D$ and write $\Nrd \colon D^\times \rightarrow F^\times$ for the reduced norm of $D$. Let $\Omega^d$ be the Drinfeld symmetric space of dimension $d$ over $K$, which is the admissible open subset of $\bP^{d, \text{an}}_K$ defined by removing all $F$-rational hyperplanes. This is stable under the natural action of $\GL_{d+1}(F)$ on $\bP^{d,\text{an}}_K$.

The Drinfeld tower is a system of rigid analytic spaces over $L$,
\[
\sM_0 \longleftarrow \sM_1 \longleftarrow \sM_2 \longleftarrow \cdots,	
\]
and each space has an action of $\GL_{d+1}(F) \times D^\times$ for which the transition morphisms are equivariant. 
For background material on these spaces, see \cite{DRI, BC, RZ}, or \cite[Section~2]{JUNEQN} for an overview. The connected components of the space $\sM_0$ are canonically identified with $\bZ$, and under this identification $(g, \delta) \in \GL_{d+1}(F) \times D^\times$ acts on this set of connected components by addition of $\nu(\det(g)\Nrd(\delta^{-1}))$. In particular, each connected component is stabilised by the subgroup
\[
	\left(\GL_{d+1}(F) \times D^\times\right)^0 \coloneqq \ker\left(\nu\left(\det(-)\Nrd\left((-)^{-1}\right)\right) \colon \GL_{d+1}(F) \times D^\times \longrightarrow \bZ\right).
\]
There is a non-canonical $\GL_{d+1}^0(F) \times \OO_D^\times$-equivariant identification of each connected component with $\Omega_{L}^d$, where 
\[
\GL_{d+1}^0(F) \coloneqq \{g \in \GL_{d+1}(F) \mid \nu(\det(g)) = 0\}
\]
and $\OO_D^\times$ acts trivially on $\Omega_L^d$. Fixing one connected component (which we will from now on identify with $\Omega_{L}^d$ as above), we can consider the preimage of $\Omega_{L}^d$ in each covering space $(\sM_n)_{n \geq 1}$. In this way we obtain a sub-tower of the full Drinfeld tower, which we denote by
\[
	\Omega_{L}^d \longleftarrow \sM_{1}^0 \longleftarrow \sM_{2}^0 \longleftarrow \cdots.
\]
Because $\Omega_{L}^d$ is stable under the action of $\GL_{d+1}^0(F) \times \OO_D^\times$, for each $n \geq 1$, $\sM_n^0 \subset \sM_n$ is $\GL_{d+1}^0(F) \times \OO_D^\times$-stable. The subgroup $1 + \Pi^n \OO_D \leq \OO_D^{\times}$ acts trivially on $\sM_{n}^0$, and the morphism $\sM_{n}^0 \rightarrow \Omega_{L}^d$ is Galois with Galois group $\OO_D^{\times} / (1 + \Pi^n \OO_D)$; see \cite[Theorem~2.2]{KOH}.
Each of the spaces $(\sM_n^0)_{n \geq 1}$ is connected over $L$, see \cite[Theorem~2.5]{KOH}, but not geometrically connected. The following result is due to Boutot and Zink and describes the connected components of $(\sM_n^0)_{n \geq 1}$ over $\bC_p$.

\begin{prop}\label{conncomp}
	There is a family of\, $\GL_{d+1}^0(F) \times \OO_D^\times$-equivariant bijections
	\[
		\pi_0\left(\sM^0_{n, \bC_p}\right)  \stackrel{\lowsim}{\longrightarrow} \frac{\OO_F^\times}{ 1 + \pi^{\lceil\frac{n}{d+1} \rceil}\OO_F}
	\]
	for any $n \geq 1$, compatible with the natural restriction maps on both sides. Here $(g,x) \in \GL_{d+1}^0(F) \times \OO_D^\times$ acts on the right by multiplication by $\det(g)\Nrd(x)^{-1} \in \OO_F^\times$.
	\end{prop}
	
	\begin{proof}
		This is \cite[Theorem~0.20]{BZ}, noting that $\Nrd(1 + \Pi^n \OO_D) = 1 + \pi^{\lceil\frac{n}{d+1} \rceil}\OO_F$; see \cite[Lemma~5]{RIEHM}.
\end{proof}

In this section we want to give a description of the $G$-invariant mod-$p$ global units of $\Omega^d$ which will use in the next section. Recall that if $R$ is a commutative ring, $\bP^d(R)$ is the set of tuples $(r_0,\ldots,r_d) \in R^{d+1}$ such that $R = Rr_0 + \cdots + Rr_{d}$, up to the scaling action $u\cdot(r_0, \ldots,r_d) = (ur_0, \ldots,ur_d)$ of $R^\times$.

\begin{defn}
	For each $m \geq 1$, let $\sH_m \coloneqq \bP^d(\OO_F / \pi^m \OO_F)$.
\end{defn}

\begin{lemma}\label{actioncor}
For all $m \geq 1$, the action of\, $G$ on $\sH_m$ is transitive.
\end{lemma}

\begin{proof}
	For notational simplicity, set $R \coloneqq \OO_F / \pi^m \OO_F$. The natural map $G \rightarrow \SL_{d+1}(R)$ is surjective because $\OO_F$ and $R$ are local rings, so both groups are generated by elementary matrices; see \cite[Theorem~4.3.9]{HOL}. The action of $\GL_{d+1}(R)$ on $\sH_m$ is transitive because any element $\mathbf{r} = (r_0, \ldots ,r_d)$ with $[\mathbf{r}] \in \sH_m$ can be extended to a basis of $R^{d+1}$, which can be seen by reducing mod-$\pi$. Then the action of $\SL_{d+1}(R)$ on $\sH_m$ is transitive as the stabiliser subgroup of the element $x = [(1 \colon 0 \colon \cdots \colon 0)]$,
	\[
		\Stab_{\SL_{d+1}(R)}(x) \leq \Stab_{\GL_{d+1}(R)}(x), 
	\]
	is of index $|R^\times|$, the same as the index of $\SL_{d+1}(R)$ in $\GL_{d+1}(R)$.
\end{proof}

\begin{defn}
For an abelian group $A$ and $m \geq 1$, we write $A[\sH_m]$ for the abelian group
\[
	A[\sH_m] \coloneqq \left\{f \colon \sH_m \longrightarrow A \right\}
\]
of all functions from $\sH_m$ to $A$, and
\[
	A[\sH_m]^0 \coloneqq \left\{f \colon \sH_m \longrightarrow A \: \middle| \: \sum_{x \in \sH_m} f(x) = 0 \right\} \subset A[\sH].
\]
\end{defn}
For any $m \geq 1$, there is a natural map
\[
	\rho_{m} \colon \sH_{m+1} \longrightarrow \sH_m,
\]
which induces
\[
	\rho_{m, *} \colon A[\sH_{m+1}] \longrightarrow A[\sH_m],
\]
defined by
\[
	\rho_{m, *}(f)(x) = \sum_{y \in \rho_m^{-1}(x)} f(y) 
\]
for all $x \in \sH_m$. This restricts to $\rho_{m, *} \colon A[\sH_{m+1}]^0 \rightarrow A[\sH_m]^0$.
\begin{defn}
	We set
\[
	A[[\sH]]^0 \coloneqq \varprojlim_{m \geq 1} A[\sH_m]^0.
\]
\end{defn}

Because each $A[\sH_m]^0$ is a $\bZ[G]$-module in a compatible way, so is $A[[\sH]]^0$. Taking $A = \bZ$, we have the following description of the global units of $\Omega^d_K$ due to Junger.

\begin{prop}[\textit{cf.} \protect{\cite[Theorem~4.5(2)]{JUNEQN}}]\label{globunitsomega}
There is an isomorphism of\, $\bZ[G]$-modules 
	\[
		\OO(\Omega^d_K)^\times / K^\times  \stackrel{\lowsim}{\longrightarrow} \bZ[[\sH]]^0.
	\]
\end{prop}

For any $m \geq 1$,
\[
|\sH_m| = q^{(m-1)d} (q^{d+1} - 1) / (q-1),
\]
and the restriction map
\[
	\rho_{m} \colon \sH_{m+1} \longrightarrow \sH_m,
\]
is surjective with each fibre of size $q^d$.

In the proof of the next lemma, we will make use of the following element.

\begin{defn}
For each $m \geq 1$, let $\Theta_m \in \bZ / p \bZ \, [\sH_m]$ be defined by
\[
	\Theta_m(x) = 1 
\]
for all $x \in \sH_m$.
\end{defn}

\begin{lemma}\label{nomodpGinvariants}
We have $(\bZ / p \bZ \, [[\sH]]^{0})^G = 0$.
\end{lemma}

\begin{proof}
	For any $m \geq 1$, we have projection maps 
	\[
		\phi_{m} \colon (\bZ / p \bZ \, [[\sH]]^0)^G \longrightarrow (\bZ / p \bZ  \, [\sH_m]^0)^G.
	\]
	Suppose that we have some $G$-invariant function, $f \in (\bZ / p \bZ \, [[\sH]]^0)^G$.
	Then for any $m \geq 1$, because $\sH_{m+1}$ is a finite set with a transitive action of $G$ (by Lemma~\ref{actioncor}),
	\[
	\phi_{m+1}(f) = \lambda \Theta_{m+1}
	\]
	for some $\lambda \in \bZ / p \bZ$. Now,
	\[
		\rho_{m} \colon \sH_{m+1} \longrightarrow \sH_m
	\]
	is surjective with each fibre of size $q^d$; hence $\phi_{m}(f) = q^d \lambda \Theta_{m} = 0$ as $p \mid q$. Therefore, $\phi_{m}(f) = 0$ for all $m \geq 1$, and hence $f = 0$.
\end{proof}

We can now use Lemma~\ref{nomodpGinvariants} to prove the main technical result of this section.

\begin{cor}\label{globalinvunitsomega}
The inclusion $K^\times \rightarrow \OO(\Omega^d)^\times$ induces an isomorphism
\[
	K^\times / K^{ \times p} \stackrel{\lowsim}{\longrightarrow} \left(\OO(\Omega^d)^\times / \OO(\Omega^d)^{ \times p} \right)^G.
\]
\end{cor}

\begin{proof}
	We have a short exact sequence of $\bZ[G]$-modules,
	\[
		0 \rightarrow K^\times \longrightarrow \OO(\Omega^d)^\times \rightarrow \bZ[[\sH]]^0 \rightarrow 0,
	\]
	and applying $- \otimes \bZ / p \bZ$, we obtain an exact sequence of abelian groups
	\begin{equation}\label{exactseq1}
		\bZ[[\sH]]^0[p] \longrightarrow K^\times / K^{\times p} \longrightarrow \OO(\Omega^d)^\times / \OO(\Omega^d)^{ \times p} \longrightarrow \frac{\bZ[[\sH]]^0}{p \bZ [[\sH]]^0} \longrightarrow 0.
	\end{equation}
	Because $p$-torsion commutes with taking the inverse limit,
	\[
		\bZ[[\sH]]^0[p] = \varprojlim_{m \geq 1} \bZ[\sH_m]^0[p] = 0.
	\]
	Furthermore, we have an exact sequence of inverse systems
	\[
		0 \longrightarrow (\bZ[\sH_m]^0)_{m \geq 1} \xrightarrow{\times p} (\bZ[\sH_m]^0)_{m \geq 1} \longrightarrow (\bZ / p \bZ \,[\sH_m]^0)_{m \geq 1} \longrightarrow 0,
	\]
	and
	\[
		\varprojlim_{m \geq 1}{}^1 \bZ[\sH_m]^0 = 0 
	\]
	because each transition map is surjective; thus the natural map
	\[
		\frac{\bZ[[\sH]]^0}{p \bZ [[\sH]]^0}  \stackrel{\lowsim}{\longrightarrow} \bZ / p \bZ \, [[\sH]]^0
	\]
	is an isomorphism. Therefore, taking the $G$-invariants of the exact sequence (\ref{exactseq1}) above, we obtain 
	\[
		0 \longrightarrow K^\times / K^{ \times p} \longrightarrow \left(\OO(\Omega^d)^\times / \OO(\Omega^d)^{ \times p} \right)^G \longrightarrow (\bZ / p \bZ \, [[\sH]]^0)^G.
	\]
	Then the conclusion follows by Lemma~\ref{nomodpGinvariants}.
\end{proof}

\section{Line bundles on the first Drinfeld covering}

Recall that we write $L$ for the completion of the maximal unramified extension of $F$ and that $K$ is a complete field extension of $F$. Let $\varpi \in \overline{F}$ be a primitive $\supst{(q-1)}$ root of $-\pi$. In this section we assume that $K$ contains $L(\varpi)$. The extension $L(\varpi)$ is the first Lubin--Tate extension of $L$, and as such is independent of the choice of $\pi$; see \cite[Theorem~3]{LT}. 

We are interested in the space $\sM_1$, which admits the following explicit description due to Junger.
\begin{defn}
	If $X$ is a rigid space over $K$, then for any $n\geq 1$, the \emph{Kummer map}
	\[
	\kappa \colon \OO(X)^\times \longrightarrow \HH^1_{\text{\'{e}t}}(X, \mu_n)
	\]
	sends $u \in \OO(X)^\times$ to
	\[
		X\left(u^{\frac{1}{n}}\right) \coloneqq \underline{\text{Sp}}_{X}(\OO_X[z] / z^n - u).
	\]
\end{defn}
Let $N \coloneqq q^{d+1} -1$ and $N' \coloneqq N / (q-1)$. In \cite[Theorem~4.9]{JUNEQN} it is shown that
\[
	\sM_1^0 \cong \Omega_L^d \left(\left(\pi u^{q-1} \right)^{\frac{1}{N}}\right) 
\]
for some particular $u \in \OO(\Omega^d)^\times$. Note that because $L$ contains all coprime to $p$ roots of $1$, $L(\varpi)$ contains a primitive $\supst{(q-1)}$ root $\tau$ of $\pi$. Therefore, over $L(\varpi)$,
\[
	\sM_{1, L(\varpi)}^0 \cong \Omega^d_{L(\varpi)} \left(\left((\tau u)^{q-1} \right)^{\frac{1}{N}}\right) \cong 
	\bigsqcup_{\zeta^{q-1} = 1} \Sigma^1_{\zeta},
\]
where
\[
	\Sigma^1_{\zeta} \coloneqq \Omega^d_{L(\varpi)}\left(\left( \zeta \tau u \right)^{\frac{1}{N'}}\right).
\]	

\begin{defn}
We let $\Sigma^1 \coloneqq \Sigma^1_{1, K}$ and let $\Sigma^2$ be the preimage of $\Sigma^1$ in $\sM_{2, K}^0$.
\end{defn}

We recall that a rigid space $X$ over a non-archimedean field $k$ is called \emph{geometrically connected} if for any finite extension $k'$ of $k$, the base change $X \times_k k'$ is connected.

\begin{cor}\label{geomconn}
The rigid spaces $\Sigma^1$ and $\Sigma^2$ are geometrically connected.
\end{cor}

\begin{proof}
Let $\Sigma^2_{1}$ be the preimage of $\Sigma^1_1$ in $\sM_{2, L(\varpi)}^0$, and let $r \in \{1,2\}$. The base change $\Sigma^r_{1, \bC_p}$ is connected by Proposition~\ref{conncomp}, noting that $\lceil\frac{r}{d+1} \rceil = 1$ because $d \geq 1$. In particular, for any finite extension $k$ of $L(\varpi)$, the base change $\Sigma^r_{1} \times_{L(\varpi)} k$ is connected, and thus $\Sigma^r_{1}$ is geometrically connected. Now $\Sigma^r_{1}$ is quasi-Stein and hence quasi-separated by \cite[Proposition~9.6.1(7)]{BGR}, and therefore by the discussion after the proof of \cite[Theorem~3.2.1]{CON}, the base change $\Sigma^r = \Sigma^r_{1} \times_{L(\varpi)} K$ is also geometrically connected.
\end{proof}

\begin{remark}
	The proof of Corollary~\ref{geomconn} shows that the $\Sigma^1_{\zeta}$ are the geometrically connected components of $\sM_{1, L(\varpi)}^0$. We note that these components are all isomorphic, as 
	\[
	\Nrd \colon \OO_D^\times / (1 + \Pi \OO_D) \longrightarrow \OO_F^\times / (1 + \pi \OO_F)
	\]
	is surjective, so the Galois group of $\sM_{1, L(\varpi)}^0 \rightarrow \Omega^d_{L(\varpi)}$ acts transitively on these components.
\end{remark}

The extension
\[
	\sM_{2,K} \longrightarrow \sM_{1, K}
\]
is Galois with Galois group
\[
H \coloneqq (1 + \Pi \OO_D) / (1 + \Pi^2 \OO_D).
\]
The extension $\Sigma^2 \rightarrow \Sigma^1$ is the restriction of this Galois covering to the open subset $\Sigma^1$ of $\sM_{1, K}$ and therefore is also Galois with Galois group $H$. From Proposition~\ref{conncomp}, we note that $\GL_{d+1}^0(F)$ acts through the determinant on the geometrically connected components of the tower, and thus $G$ stabilises both $\Sigma^1$ and $\Sigma^2$. Furthermore, the action of $G$ on both $\Sigma^1$ and $\Sigma^2$ commutes with the Galois action.

\begin{prop}\label{globalinvunitssigma}
The inclusion $K^\times \rightarrow \OO(\Omega^d)^\times$ induces an isomorphism
	\[
		K^\times / K^{ \times p}\stackrel{\lowsim}{\longrightarrow} \left(\OO(\Sigma^1)^\times / \OO(\Sigma^1)^{ \times p} \right)^G.
	\]
\end{prop}

\begin{proof}
	Let $\sigma$ be a primitive $\supth{N}$ root of $\pi$. Then by \cite[Theorem~5.1]{JUNEQN}, there is a short exact sequence of abelian groups
	\[
		0 \longrightarrow \OO\left(\Omega^d_{K(\sigma)}\right)^\times \longrightarrow \OO\left(\Sigma^1_{K(\sigma)}\right)^\times \longrightarrow \bZ / (q+1) \bZ \longrightarrow 0.
	\]
	Taking $\Gal(K(\sigma) / K(\varpi))$-invariants and applying $- \otimes \bZ / p \bZ$, we are left with an isomorphism
	\[
		\OO\left(\Omega^d\right)^\times / \OO\left(\Omega^d\right)^{ \times p} \stackrel{\lowsim}{\longrightarrow} \OO\left(\Sigma^1\right)^\times / \OO\left(\Sigma^1\right)^{ \times p}.
	\]
	The result then follows from Corollary~\ref{globalinvunitsomega}.
\end{proof}

We now want to show that the homomorphism
\[
\widehat{H} \longrightarrow \Pic(\Sigma^1)[p]
\]
associated to the Galois covering $f \colon \Sigma^2 \rightarrow \Sigma^1$ is injective. In order to prove this, we will make use of the following explicit description of the Kummer exact sequence.

Recall that if $X$ is a rigid space over $K$, then for any $n \geq 1$, the Kummer exact sequence is the short exact sequence
\[
	0 \longrightarrow \OO(X)^\times / \OO(X)^{\times n} \longrightarrow \HH^1_{\et}(X, \mu_n) \longrightarrow \Pic(X)[n] \longrightarrow 0 
\]
arising from the long exact sequence of the functor $\Gamma(X_{\et}, -)$ applied to the sequence
\[
	0 \longrightarrow \mu_n \longrightarrow \bG_m \xrightarrow{\times n} \bG_m \longrightarrow 0 
\]
of sheaves of $X_{\et}$, which is exact because $n$ is invertible in $K$; see \cite[Section~3.2]{PJ}. There is a more explicit description of this sequence, which we now summarise. References in the case of schemes are \cite[Tag 03PK]{STACK} and \cite[Section~III.4]{MIL}, from which the case for rigid spaces can be deduced \textit{mutatis mutandis}.

Let $\{(\sL, \alpha)\} / \! \cong$ be the set of pairs $(\sL, \alpha)$, where $\sL \in \Pic(X)$ and $\alpha \colon \sL^{\otimes n} \xrightarrow{\lowsim} \OO_X$ is an $\OO_X$-linear isomorphism, considered up to the natural notion of isomorphism. The set $\{(\sL, \alpha)\} /\! \cong$ forms an abelian group, and there is an isomorphism of short exact sequences 
\[\begin{tikzcd}
	0 & {\OO(X)^\times / \OO(X)^{\times n}} & {\{(\sL, \alpha)\} / \! \cong} & {\Pic(X)[n]} & 0 \\
	0 & {\OO(X)^\times / \OO(X)^{\times n}} & {\HH^1_{\et}(X, \mu_n)} & {\Pic(X)[n]} & 0
	\arrow[from=2-1, to=2-2]
	\arrow[from=2-2, to=2-3]
	\arrow[from=2-3, to=2-4]
	\arrow[from=2-4, to=2-5]
	\arrow[from=1-4, to=1-5]
	\arrow[from=1-3, to=1-4]
	\arrow[from=1-2, to=1-3]
	\arrow[from=1-1, to=1-2]
	\arrow["{=}"', from=1-2, to=2-2]
	\arrow["\sim"', from=1-3, to=2-3]
	\arrow["{=}"', from=1-4, to=2-4]
\end{tikzcd}\]
The homomorphism $\{(\sL, \alpha)\} /\!  \cong \: \rightarrow \Pic(X)[n]$ is simply $[(\sL, \alpha)] \mapsto [\sL]$.
Given a pair $[(\sL, \alpha)]$,  the associated $\mu_n$-torsor in $\HH^1_{\et}(X, \mu_n)$ is $Z \coloneqq \underline{\Sp}(\sA)$, where $\sA$ is the coherent sheaf of $\OO_X$-algebras
\[
	\sA = \bigoplus_{i = 0}^{n-1} \sL^{\otimes i},
\]
with multiplication the natural maps
\[
\begin{array}{ll}
	\sL^{\otimes i} \otimes \sL^{\otimes j} \longrightarrow \sL^{i + j} & \quad \mbox{if } i + j \leq n-1, \\
	\sL^{\otimes i} \otimes \sL^{\otimes j} \longrightarrow \sL^{i + j} \stackrel{\alpha}{\longrightarrow} \sL^{i + j - n} & \quad \mbox{if } i+j \geq n
\end{array}
\]
for $0 \leq i,j \leq n$. In order to describe the structure of $Z$ as a $\mu_n$-torsor, we first consider this construction locally.

Suppose that $\sL = \OO_X$. In this case the isomorphism $\alpha$ has the form $\alpha \colon \OO_X^{\otimes n} \rightarrow \OO_X$, and we can use the canonical isomorphism $\psi \colon \OO_X \rightarrow \OO_X^{\otimes n}$ to define $a \coloneqq \alpha(\psi(1)) \in \OO_X(X)^\times$. Then under the construction above,
\[
Z = \underline{\Sp}(\OO_X[z] / (z^n - a)).
\]
For any rigid space $Y$ over $X$, $Z(Y) = \{s \in \OO_Y(Y) \mid s^n = a\}$, which has the structure of a $\mu_n$-torsor via
\[
	\mu_n(Y) \times Z(Y) \longrightarrow Z(Y), \quad (\zeta, s) \longmapsto \zeta s.
\]
Now for a general pair $[(\sL, \alpha)]$, the associated space $Z$ is locally in the rigid topology of the above form, and these structures patch to give $Z$ the structure of a $\mu_n$-torsor.

Now suppose that $K$ contains a primitive $\supth{n}$ root of $1$. In this case the group scheme $\mu_n$ is naturally identified with the constant group scheme $\underline{\mu_n(K)}$, and under this identification there is a correspondence between $\mu_n$-torsors and Galois coverings $Z \rightarrow X$ with Galois group $\underline{\mu_n(K)}$ (to use the language of Section~\ref{section1}).

We are interested in the homomorphism $\HH^1_{\et}(X,\mu_n) \rightarrow \Pic(X)[n]$. From the description of the $\mu_n$-action above, we see that if a Galois covering $f \colon Z \rightarrow X$ corresponds to the pair $[(\sL, \alpha)]$, we can recover $\sL$ as the line bundle
\[
	\sL \cong e_{\iota} \cdot f_*\OO_Z,
\]
where $\iota$ is the natural inclusion $\iota \colon \mu_n(K) \rightarrow K^\times$. More generally, if $f \colon Z \rightarrow X$ is a Galois covering with Galois group $\Gamma$ and $\chi \colon \Gamma \xrightarrow{\lowsim} \mu_n(K)$ is an isomorphism, then in the induced exact sequence
\[
	0 \longrightarrow \OO(X)^\times / \OO(X)^{\times n} \longrightarrow \HH^1_{\et}(X, \underline{\Gamma}) \longrightarrow \Pic(X)[n] \longrightarrow 0,
\]
the image of the Galois covering $f \colon Z \rightarrow X$ in $\Pic(X)[n]$ is the line bundle $e_{\chi} \cdot f_* \OO_{Z}$.

\begin{thm}\label{mainthm}
Suppose that $K$ contains $L(\varpi)$ and a primitive $\supth{p}$ root of\, $1$. Then the homomorphism
\[
	\widehat{H} \longrightarrow \Pic(\Sigma^1)[p]^G, \quad \chi \longmapsto \sL_{\chi} = e_{\chi} \cdot f_*\OO_{\Sigma^2}
\]	
is injective.
\end{thm}

\begin{remark}
	The assumption that $K$ contains $L(\varpi)$ is simply to ensure the space $\Sigma^1$ is defined, and the assumption that $K$ contains a primitive $\supth{p}$ root of $1$ is similarly to ensure that the homomorphism is defined. Furthermore, when $F$ is unramified, the assumption that $K$ contains a primitive $\supth{p}$ root of $1$ in the statement of Theorem~\ref{mainthm} is superfluous. Indeed, $K$ contains $L(\varpi)$, and the Lubin--Tate extensions $\bQ_p(\zeta_p)$ and $\bQ_p((-p)^{1 / (p-1)})$ of $\bQ_p$ are equal.
\end{remark}

\begin{proof}
	Let $\chi \colon H \rightarrow K^\times$ be non-trivial. We want to show that $e_{\chi} \cdot f_*\OO_{\Sigma^2} \in \Pic(\Sigma^1)$ is non-trivial. Because $H$ has exponent $p$ and $\chi$ is non-trivial, $\chi$ induces an isomorphism
	\[
		\chi' \colon H / H_{\chi} \stackrel{\lowsim}{\longrightarrow} \mu_p(K),
	\]
	where $H_{\chi}$ is the kernel of $\chi$. From $H_{\chi}$ we may form the quotient 
	\[
		f' \colon \Sigma^2 / H_{\chi} \longrightarrow \Sigma^1.
	\]
	If $U \subset \Sigma^1$ is an admissible open subset and $V = f^{-1}(U) \subset \Sigma^2$, then above $U$ the quotient $\Sigma^2 / H_{\chi}$ is described by $\Sp(\OO(V)^{H_{\chi}})$. Because $H_{\chi}$ is normal, $f' \colon \Sigma^2 / H_{\chi} \rightarrow \Sigma^1$ is Galois with Galois group $H / H_{\chi}$, which follows from \cite[Theorem~2.2]{CHR} and the fact that each property in the definition of a Galois extension checked affinoid locally.
	
	We first note that we have an equality of $\OO_{\Sigma^1}$-modules
	\[
		e_{\chi} \cdot f_*\OO_{\Sigma^2} = e_{\chi'} \cdot f'_* \OO_{\Sigma^2 / H_{\chi}}.
	\]
	Indeed, for any admissible open subset $U$ of $\Sigma^1$,
	\[
		(e_{\chi'} \cdot f'_* \OO_{\Sigma^2 / H_{\chi}})(U) = e_{\chi'} \cdot \OO_{\Sigma^2}(f^{-1}(U))^{H_{\chi}}
	\]
	and
	\[
		(e_{\chi} \cdot f_*\OO_{\Sigma^2})(U) = e_{\chi} \cdot \OO_{\Sigma^2}(f^{-1}(U)).
	\]
	Setting $B \coloneqq \OO_{\Sigma^2}(f^{-1}(U))$, we have that
	\begin{align*}
		e_{\chi}\cdot B &= \{b \in B \mid h(b) = \chi(h)b \text{ for all } h \in H\}, \\
		e_{\chi'} \cdot B^{H_{\chi}} &= \{b \in B^{H_{\chi}} \mid h(b) = \chi(h)b \text{ for all }h \in H / H_{\chi}\},
	\end{align*}
	and it is direct to check that these are equal. Therefore, we are reduced to showing that $e_{\chi'} \cdot f'_* \OO_{\Sigma^2 / H_{\chi}}$ is non-trivial.

	Now because the action of $G$ on $\Sigma^2$ and $\Sigma^1$ commutes with the action of $H$, $G$ acts on $\Sigma^2 / H_{\chi}$, $f' \colon \Sigma^2 / H_{\chi} \rightarrow \Sigma^1$ is $G$-equivariant, and the $G$-action commutes with the action of $H / H_{\chi}$. Therefore, the covering $f' \colon \Sigma^2 / H_{\chi} \rightarrow \Sigma^1$ defines an element of 
	$\HH^1_{\et}(\Sigma^1, \underline{H / H_{\chi}})^G$, see \cite[Section~4.1]{JUNEQN}, the middle term of the $G$-invariants of the Kummer exact sequence
	\begin{equation}\label{Ginvkummer}
		0 \longrightarrow \left( \OO(\Sigma^1)^\times / \OO(\Sigma^1)^{\times p} \right)^G \longrightarrow \HH^1_{\et}(\Sigma^1, \underline{H / H_{\chi}})^G \longrightarrow \Pic(\Sigma^1)[p]^G.
	\end{equation}
	Now suppose towards a contradiction that the line bundle $e_{\chi'} \cdot f'_* \OO_{\Sigma^2 / H_{\chi}}$ is trivial. Then from the exact sequence (\ref{Ginvkummer}) above, the space $\Sigma^2 / H_{\chi}$ is given as $\kappa(v) = \Sigma^1(v^{1/p})$ for some
	\[
		v \in \left( \OO(\Sigma^1)^\times / \OO(\Sigma^1)^{\times p} \right)^G.
	\]
	By Proposition~\ref{globalinvunitssigma}, we actually have $v \in K^\times / K^{\times p}$, and therefore the base change $\Sigma^2 / H_{\chi} \times_K K({v}^{1/p})$ is not connected. Over $K({v}^{1/p})$, the intermediate extension
	\[
	\Sigma^2 \times_K K({v}^{1/p}) \longrightarrow (\Sigma^2 / H_{\chi}) \times_K K({v}^{1/p})
	\]
	is Galois and hence surjective, and thus $\Sigma^2 \times_K K({v}^{1/p})$ is also not connected. But this gives a contradiction as $\Sigma^2$ is geometrically connected by Corollary~\ref{geomconn}.
\end{proof}

\begin{remark}
If we do not assume that $K$ contains a primitive $\supth{p}$ root of $1$, then the techniques used in the proof of Theorem~\ref{mainthm} can still be used to show that $\Pic(\Sigma^1)[p]^G \neq 0$. Indeed, if we assume that $\Pic(\Sigma^1)[p]^G = 0$, then the same argument but with $H_{\chi}$ replaced by any index $p$ subgroup $H_0$ of $H$ will still result in a contradiction.
\end{remark}

\section{Vector bundles on the Drinfeld upper half plane}

In this section we provide an elementary proof that all vector bundles on $\Omega^1$ are trivial, which extends and uses the result that all line bundles on $\Omega^1$ are trivial; see \cite[Theorem~A]{JUNCOH}. In the context of Theorem~\ref{mainthm}, this says that whilst the line bundles $\sL_{\chi}$ on $\Sigma^1$ are non-trivial whenever $\chi \neq 1$, the pushforward to $\Omega^1$ will be a trivial vector bundle (of constant rank $q+1$). 

Before we state the theorem, we will need the following notions from commutative algebra.

\begin{defn}
An integral domain $R$ is called a \emph{Pr\"{u}fer domain} if every finitely generated ideal of $R$ is invertible; it is called a \emph{B\'{e}zout domain} if every finitely generated ideal of $R$ is principal. 
\end{defn}

We provide a proof of the following result, for which we were unable to find a reference.

\begin{lemma}\label{bezoutlemma}
Suppose that $R$ is a B\'{e}zout domain. Then every finitely generated submodule of a free module is free.
\end{lemma}

\begin{proof}
	Suppose that $M$ is finitely generated over $R$, and $M$ is contained in a free module $P$. By choosing a basis for $P$, as $M$ is finitely generated, we have that $M \subset R^n$ for some $n \geq 1$. Let $\pi \colon R^n \rightarrow R$ be the projection to the first factor, and let $I \coloneqq \pi(M)$ and $K \coloneqq \ker(\pi \colon M \rightarrow R)$. Now $I$ is the homomorphic image of $M$ and thus finitely generated; hence $I$ is principal and thus free because $R$ is a B\'{e}zout domain. Therefore, the short exact sequence
	\[
		0 \longrightarrow K \longrightarrow M \longrightarrow I \longrightarrow 0
	\]
	splits, and $M \cong K \oplus I$. Finally, $K$ is also finitely generated, being a homomorphic image of $M$, and $K \subset R^{n-1}$, so the result follows by induction.
\end{proof}

We note that this property actually characterises B\'{e}zout domains among integral domains. Indeed, if $I$ is a finitely generated ideal of an integral domain $R$ which satisfies the above property, then $I$ is free, but also $I \subset R$, so by passing to the fraction field of $R$, we see that $I$ must have rank $1$, and thus $I$ is principal. This property is analogous to the following property of PIDs (which are exactly the Noetherian B\'{e}zout domains): a commutative ring $R$ is a PID if and only if every submodule of a free module is free.

\begin{thm}
Let $\fX$ be a smooth connected $1$-dimensional quasi-Stein rigid analytic space, with $\Pic(\fX) = 0$. Then any vector bundle on $\fX$ is of the form $\OO_{\fX}^n$ for some $n \geq 0$.
\end{thm}

\begin{proof}
  If $\fX$ is as above, the ring $R \coloneqq \OO_{\fX}(\fX)$ is an integral domain. The global sections functor defines an equivalence of categories between vector bundles on $\fX$ and finitely generated projective modules over $R$; see \cite[Proposition~1.13]{BSX}.
  In particular, $\Pic(R) = 0$, and we are reduced to showing that any finitely generated projective module over $R$ is free. The ring $R$ is a Pr\"{u}fer domain, see \cite[Corollary~1.8]{BSX},
  and because $\Pic(R) = 0$, the ring $R$ is furthermore a B\'{e}zout domain. Then we can conclude, as for such rings any finitely generated projective module is free, by Lemma~\ref{bezoutlemma}.
\end{proof}

\begin{cor}\label{cor2}
Any vector bundle on $\Omega^1$ is of the form $\OO_{\Omega^1}^n$ for some $n \geq 0$. 
\end{cor}


\end{document}